\documentclass[a4paper,11pt]{amsart}
\usepackage{amsmath,amssymb,colordvi}
\usepackage{graphicx}

\newtheorem{theorem}{Theorem}

\newtheorem{example}[theorem]{\it Example}
\newtheorem{proposition}[theorem]{Proposition}

\newtheorem{definition}[theorem]{Definition}

\newtheorem{remark}[theorem]{\it Remark}

\newcommand{\be}{\begin{enumerate}}
\newcommand{\ee}{\end{enumerate}}

\font\tenBb=msbm10 \font\sevenBb=msbm7 \font\fiveBb=msbm5
\newfam\Bbfam \textfont\Bbfam=\tenBb \scriptfont\Bbfam=\sevenBb
\scriptscriptfont\Bbfam=\fiveBb
\newcommand{\Q}{\mathcal{Q}}
\def\Bb{\fam\Bbfam\tenBb}

\def\Q{{\Bb Q}}
\def\I{{\Bb I}}

\def\A{{\mathcal A}}

\def\Q{{\mathcal Q}}

\everymath{\displaystyle}
\font\dsrom=dsrom10 scaled 1200 
\def \ind{\textrm{\dsrom{1}}}

\begin{document}
\title{New copulas and their applications to Symmetrizations  of bivariate copulas}
\maketitle

 \author{Mohamed El maazouz}
  \address{Universit\'e Ibn Zohr, Facult\'e des Sciences, D\'epartement de Math\'ematiques, Agadir, Maroc}
 \email{maaz71@gmail.com }

\author{Ahmed Sani }
\address{Universit\'e Ibn Zohr, Facult\'e des Sciences, D\'epartement de Math\'ematiques, Agadir, Maroc}
\email{ahmedsani82@gmail.com }

\thanks{}

\footnote{\today}

\subjclass[\textbf{Subject classification 2010}]{Primary   62 Hxx;62 Exx Secondary 46 Axx  }

\keywords \textbf{Keywords:}{asymmetric copulas, perturbation}

\vskip 2 cm

\begin{center}
\textsf{Abstract}
\end{center}
New copulas, based on perturbation theory, are introduced to clarify a \emph{symmetrization} procedure for  asymmetric copulas. We  give also some properties of the \emph{symmetrized} copula. Finally, we examine families of copulas with a prescribed  symmetrized one. By the way, we study topologically, the set of all symmetric copulas and give some of its classical and new properties.

\section*{Introduction}

The notion of non ex-changeability was first  introduced by Hollander \cite{Ho71} to quantify the asymmetry of factors which describe a given phenomena. It was baptized \emph{asymmetry} by Nelsen in \cite{Ne93},\cite{Ne06} and \cite{NE07},  Mesiar and  Klement adopted the same nomenclature in \cite{KM06}, so  all other researchers in this emergent field of copulas pursue how to detect, to measure and sometimes to avoid the asymmetry. In a general context, the ex-changeability relies on the asymmetry procedures with non commutative effect in different directions. In chemistry, for example, the dosage acid/base and  reciprocally base/acid do not lead to same $pH-$ aqueous solution. In economical point of view, the respective combination of $C(k,l)$ of $k$  and $l$  units of capital and labor do not ensure the same \emph{utility} as the $C(l,k)$ ensures. Other examples abound in many fields to show the importance of asymmetry. A notion closely related to the topic is the radial symmetry as treated first in \cite{Ne06} and developed later in \cite{Deh13}.
  
For statistical applications, multiple regression theory leads in a general framework to estimation of parameters as critical points of first order condition. For the simple case $Y=X\beta+\epsilon $ when some classical hypotheses are satisfied, an estimator of $\beta$ is given by $\hat{\beta}=^{t}XX^{t}XY$. For a symmetric matrix $X$, the calculus become  easier since, in adapted basis, $^{t}XX=X^2$.\\
    The  concept of \emph{symmetrization } consists in, as it will be developed below, a classical and simple decomposition of a copula that governs the joint distribution of random variables. Copulas gained a great interest in recent decades because of their relatively simple use in statistics and probability. Their importance comes back to the historical Sklar's theorem (see \cite{Sk59},\cite{Sk73} or the unavoidable Nelsen's book \cite{Ne06}). The Sklar result presents the copula as a natural bridge between margin distributions of  random vector and its joint law. A recent and topological new proof of this important result was given by Durante et al.\cite{Durante13sklar}.
In the first section of the current paper, many results on copulas that we will need are given as preliminaries. We also definitely make precise notations and recall some recent developments on asymmetry. We recall some results on classical dependence parameters expressed in terms of copulas, see\cite{Bu20}.\\

 The second section is devoted to define the \emph{symmetrized} and\emph{ the radial symmetrized} copulas. We give some properties linking a given copula and its symmetrized and radial symmetrized ones in the same spirit of \cite{SMP20}.\\

 In the last section, we study  the inverse problem by giving new copulas as solutions of 
 some functional equations. The new copulas are based on an asymmetric perturbation of the independence copula 
 $\Pi$. More precise, for a given symmetric or radial symmetric copula $S$, we determine copulas for which  $S$ is the prescribed \emph{symmetrized} or the \emph{radial symmetrized} part.


Along this paper, $\I$ denotes the closed interval $[0,1]$ and $\|.\|$ the usual uniform  norm  on the set $\mathcal{C}$ of all bivariate copulas. Let $\mathcal{S}$ and $\mathcal{R}$ denote respectively its subsets of  all symmetric and  radial symmetric copulas. 
 
\section{Preliminaries}

  \begin{definition}\label{def1}
  	A copula $C$ is a bifunction on $\I^2$ into $\I$ which satisfies the following conditions for all $u,v,u_1,v_1,u_2,v_2$ in $\I$
  	\begin{enumerate}
  		\item Border conditions: $C(0,v)=C(u,0)=0.$
  		\item Uniform margins: $C(1,v)=v$ and $C(u,1)=u.$
  		\item the $C-$volume property: $V_C(R)= C(u_2,v_2)-C(u_2,v_1)-C(u_1,v_2)+C(u_1,v_1)\geq 0$ for all rectangle $R=[u_1,v_1]\times[u_2,v_2] \subset \I^2$ with $u_1<u_2$ and $u_2<v_2$.
  	\end{enumerate}
 \end{definition}
 Any element of $\mathcal C$ is  framed between Fr\'echet-Hoeftding bounds   $W$ and $M$ given by $W(u,v)=\max(u+v-1,0)$ and $M(u,v)=\min(u,v)$. Precisely, we have
 $$
 \forall (x,y)\in \I^2:\quad W(x,y)\le  C(x,y)\le M(x,y). 
 $$
  These bounds ($M$ and $W$)  are also copulas but in higher dimensions, say for multivariate copulas, $W$ is not a copula. \\
   Statistically speaking,  Fr\'echet-Hoeftding bounds  $M$ and $W$ model respectively the co-monotonicity and anti-monotonicity of  empirical variables $X$ and  $Y$. The copula $\Pi:\ (x,y)\in \I^2\mapsto xy$ characterizes the total independence between the two variables.\\
  	
On the other hand, some derived copulas from a given one will serve as efficient tool to treat the asymmetry questions  mainly those of   \emph{transpose copula}, \emph{survival } one.
 
  
 \begin{definition}

  	The survival copula of a given copula $C$
  	 is the function $\widehat{C}$ defined by the formula $$\forall u,v\in \I:\quad \widehat{C}(u,v)=u+v-1+C(1-u,1-v)$$
  \end{definition}

we denote by $\Q$  (\cite[pages,159-182]{Ne06}) the difference of two probabilities $\Q = P((X_1-X_2)(Y_1-Y_2) > 0) - P((X_1 - X_2)(Y_1 - Y_2) < 0)$
for a pair of continuous random vectors $(X_1, X_2)$ and $(Y_1, Y_2)$. If the corresponding copulas are $C_1$ and $C_2$, then we have $$\Q(C_1,C_2)=4\int_{\I^2}C_1(u,v)~~dC_2(u,v)-1$$

The usual four measures of concordance may be defined in terms of the concordance function $\Q$.
\\Kendall’s tau of $C$ is defined by
$$\tau(C) = \Q(C, C)$$ 
Spearman’s rho by
$$\rho(C)=3\Q(C, \Pi)$$
Gini’s gamma by
$$\gamma(C) = \Q(C, M) + \Q(C, W).$$ 
Blomqvist's Beta by 
$$\beta(C)=4C\left(\frac{1}{2},\frac{1}{2}\right)-1.$$
%
%
%

\begin{theorem}
Let $X$ and $Y$ be continuous random variables with distribution functions $F$ and $G$, respectively, 
and let $C$ be the copula of $X$ and $Y$. If theses
limits  exist, then the upper and lower tail dependence are given as:

$$\lambda_U=2-\underset{t\rightarrow 1^-}{\lim}\frac{1-C(t,t)}{1-t}~~\text{
and}~~ \lambda_L=\underset{t\rightarrow 0^+}{\lim}\frac{C(t,t)}{t}.$$
\end{theorem}

\section{Asymmetry and radial asymmetry of bivariate copulas}
 
\begin{definition}
A bivariate copula $C$ is said:
\begin{itemize}
\item[•]Symmetric  if it satisfies $$\forall u,v\in \I:\quad C(u,v)=C(v,u).$$
\item[•]Radial symmetric  if it satisfies $$\widehat{C}=C.$$

\end{itemize}
 
We denote $C^T$ the transpose of $C$ defined by $C^T(u,v)=C(v,u)$.
\end{definition}
\begin{example}
\begin{enumerate}
\item The bounds $M$ and $W$ of Fr\'echet-Hoeffding and the independence copula $\Pi$ are symmetric and also radial symmetric.
\item   Let $\alpha$ and $\beta$ be in $(0,1)$. The Marshall Olkin family of copulas.  $$C_{\alpha,\beta}(u,v)=min(u^{1-\alpha}v,uv^{1-\beta})$$ is asymmetric if $\alpha \neq \beta$.

\end{enumerate} 
\end{example}
\begin{proposition}
The subsets  $\mathcal{S}$ and $\mathcal{R}$ are both convex and  closed  in $\mathcal{C}$ with respect to the uniform convergence. 
\end{proposition}

\begin{proof}
Let $(C_n)$ be a sequence of symmetric copulas which converges uniformly to a copula $C$. We show that $C$ is also symmetric. For all u,v in $\I$, we have:
$$ C(v,u)=\underset{n\rightarrow +\infty}{\lim}C_n(v,u)=\underset{n\rightarrow +\infty}{\lim}C_n(u,v)=C(u,v)$$ 
Let now $(C_n)$ be a sequence in $\mathcal{R}$ which converges uniformly to a copula $C$.
$$C(u,v)=\underset{n\rightarrow +\infty}{\lim}C_n(u,v)=\underset{n\rightarrow +\infty}{\lim}\widehat{C_n}(u,v)=\widehat{C}(u,v)$$ 
If $C_1,C_2\in \mathcal{S}$ and $\lambda$ is a real number in $\I$, then easily $\lambda C_1+(1-\lambda)C_2$ belongs to $\mathcal{S}$
\\If $C_1,C_2\in \mathcal{R}$ and $\lambda$ is a real number  in $\I$, then  $\lambda C_1+(1-\lambda)C_2 \in \mathcal{R}$.
\end{proof}
\begin{proposition}\label{Projection}
\begin{itemize}
\item[•]The projection of any bivariate copula $C$ respect to $\mathcal{S}$ is the symmetric copula $C_S=\frac{C+C^T}{2}$.
\item[•]The projection of any bivariate copula $C$ respect to $\mathcal{R}$ is the radial symmetric copula $C_R=\frac{C+\widehat{C}}{2}$.
\end{itemize}

\end{proposition}
\begin{proof}
We show that  $\frac{C+C^T}{2}$ is the closest symmetric copula  to $C$.
\\Let $S$ be any symmetric copula. Since $$\forall (u,v)\in \I^2,~~|C(u,v)-S(u,v)|=|C^T(v,u)-S^T(v,u)|=|C^T(v,u)-S(v,u)|.$$,  we have $\|C-S\|=\|S-C^T\|.$
And this equality gives 
$$\|C-C^T\|\leq \|C-S\|+\|S-C^T\|\leq 2\|C-S\|$$
so $$\|\frac{C-C^T}{2} \|\leq \|C-S\|$$
\\A similar proof can be given to the second point. 
\end{proof}

\begin{remark}
A measures $\mu$ of asymmetry ($\nu$ of radial asymmetry) of a copula $C$ can be  defined naturally as the distance between $C$ and $\mathcal{S}$ ($C$ and $\mathcal{R}$). So  
 $$\mu(C)=\|\frac{C-C^T}{2}\|$$
and respectively 
$$\nu(C)=\|\frac{C-\widehat{C}}{2}\|$$
\end{remark}

In the sequel, we give some obvious but important consequences of this splitting procedure of copulas to (radial) symmetric and (radial) asymmetric parts.

\begin{proposition}
\be
\item Let  $C$ be a copula. We have 
$$\left( C^T\right)_S=C_S,~~\left( \widehat{C}\right)_R=C_R,~~\widehat{C_S}=\left(\widehat{C}\right),~~(C_R)^T=(C^T)_R$$

\item The copulas $C$ and $C_S$ have same  Spearman $\rho$,  same Gini's gamma and  same Blomqvist's beta. But not necessarily  same Kendhal's tau.
\item The copulas $C$ and $C_R$ have same  Spearman $\rho$,  same Gini's gamma and  same Blomqvist's beta. But not necessarily  same Kendhal's tau.
\item If a copulas $C$ have tail dependence parameters, then $C_S$ have same upper and lower tails dependence parameter than $C$. Meaning:
$$\lambda_L(C)=\lambda_L(C_S)~~\text{and}~~\lambda_U(C)=\lambda_U(C_S)$$
\item If a copulas $C$ have tail dependence parameters, then also for the radial symmetrized $C_R$ and we have $$\lambda_L(C_R)=\lambda_U(C_R)=\frac{1}{2}\left( \lambda_L(C)+\lambda_U(C)\right)$$

\ee

\end{proposition}

\begin{proof}
\begin{enumerate}
 \item  Theses equalities  are  an  immediate consequences of  definitions.

\item
%
\begin{itemize}

\item[•] 
$$\begin{tabular}{ccc}
$2\int_{\I^2}C_S(u,v)~~dudv$ & $=$ & $\int_{\I^2}(C+C^T)(u,v)~~dudv$\\
$ $ & $=$ & $\int_{\I^2}C(u,v)~~dudv+\int_{\I^2}C^T(u,v)~~dudv$\\
$ $ & $=$ & $2\int_{\I^2}C(u,v)~~dudv$
\end{tabular}$$
And this last calculation leads to $\rho(C_S)=\rho(C)$.
\item[•] By use of symmetry of $M$ and $W$, we have also $\mathcal{Q}(C,M)=\mathcal{Q}(C^T,M)$ and $\mathcal{Q}(C_S,W)=\mathcal{Q}(C^T,W)$ so $\gamma(C_S)=\gamma(C)$
\item[•] From the equality $C_S\left(\frac{1}{2},\frac{1}{2}\right)=C\left(\frac{1}{2},\frac{1}{2}\right)$ we conclude that $\beta(C_S)=\beta(C)$
\item[•] The equality between the Kendal tau of $C$ and $C_S$ is not generally satisfied. Consider the Marshal Olkin copula $C_{\frac{1}{2},\frac{1}{4}}$ for a counter example. 
$$C(u,v)=u^{\frac{1}{2}}v\ind_{u\leq v^2}(u,v)+uv^{\frac{3}{4}}\ind_{u>v^2}(u,v)\;\;~~~~\tau(C)=\frac{1}{5}~~\text{See \cite[165]{NE07} }$$

$$\begin{tabular}{ccc}
$\int_{\I^2}C^T dC$ & $=$ &$\int_{v\leq u^2}C^T dC+\int_{u^2\leq v \leq  \sqrt{u}}C^T dC+\int_{\sqrt{u}\leq v}C^T dC$\\
$ $ & $=$ & $\frac{3}{4}\int_{v\leq u^2}uv^{\frac{1}{4}} ~~du dv+\frac{3}{4}\int_{u^2\leq v \leq  \sqrt{u}}(uv)^{\frac{3}{4}}+\frac{1}{2}\int_{\sqrt{u}\leq v}u^{\frac{1}{4}}v~~du dv$\\
$ $ & $=$ & $\frac{3}{4}\int_{0}^{1}u\left(\int_{0}^{u^2}v^{\frac{1}{4}} dv\right)~~du +\frac{3}{4}\int_{0}^{1}u^{\frac{3}{4}}\left(\int_{u^2}^{\sqrt{u}}v^{\frac{3}{4}} dv\right)~~du+\frac{1}{2}\int_{0}^{1}u^{\frac{1}{4}}\left(\int_{\sqrt{u}}^{1}v dv\right)~~du$\\
$ $ &$=$ &$\frac{402}{1323}$
\end{tabular}$$
Hence
$$\begin{tabular}{ccc}
$\int_{\I^2}C_S(u,v)dC_S$ & $=$ & $\frac{1}{4}\int_{\I^2}\left(C+C^T)(dC+dC^T)\right)$\\
$ $ & $=$ & $\frac{1}{2}\int_{\I^2}C dC+\frac{1}{2}\int_{\I^2}C^T dC$\\
$ $ & $=$ & $\frac{1}{2}\left(\frac{1+\tau_C}{4}\right)+\frac{1}{2}\int_{\I^2}C^T dC\neq \frac{1}{5}$
\end{tabular}$$

\end{itemize}
\item It suffices to see that $$\mathcal{Q}(\widehat{C},\Pi)=\mathcal{Q}(C,\Pi), ~~\mathcal{Q}(\widehat{C},M)=\mathcal{Q}(C,M)~~\text{and}~~ \mathcal{Q}(\widehat{C},W)=\mathcal{Q}(C,W)$$
\item It suffices to remark that $\forall t\in \I:~~C(t,t)=C_S(t,t)$.
\item An obvious calculus leads to  the result.

\end{enumerate}
\end{proof}
As a natural consequence, all copulas in the convex envelope of $C$, $C_S$ and $C_R$ have the same  Spearman's coefficient $\rho$, the same Blomqvist's Beta and the same Gini's Gamma.


A natural question is the \emph{inverse problem} in the sense that
 given a symmetric copula $S$, does it exist a copula $C$ other than $S$ itself such that $C_S=S$? In the following section, we give an answer to this problem for  some remarkable copulas mainly the comprehensive ones: $M$, $W$ and $\Pi$.
 
\section{New copulas and inverse problem}
 
 As mentioned above, the sequel is devoted to the inverse problem as explained at the end of the last section. To do so, we begin by defining a new copula based on perturbation theory:
 
 \begin{theorem}\label{newcopula}
 For all $(u,v)\in \I^2$, we put: $P(u,v)=(u-v)\min(u,v)\min(1-u,1-v)$. Then, the mapping
 $$C(u,v)=\Pi(u,v)+P(u,v)=uv+(u-v)\min(u,v)\min(1-u,1-v)$$
 defines a copula.
\end{theorem}
 
 The mapping $C$ is a natural perturbation of independence copula $\Pi$. One may se \cite{Durante13}  for more detail on perturbation theory of copulas. The novelty in this new copula is, for the best of our knowledge, its asymmetry, contrary to the most known perturbed copulas in the literature.\\
 Let us prove that $C$ is actually a copula
 
 \begin{proof}[of Theorem (\ref{newcopula})]
 $P$ satisfies for all $u,v\in \I$ 
$$P(v,u)=-P(u,v), P(u,0)=P(u,1)=P(0,v)=P(1,v)=0$$
Thus the bifunction $C=\Pi+P$ satisfies the  copula boundary conditions.
 \\Let $R=[u_1,u_2]\times[v_1,v_2]$ be a rectangle in $\I^2$.
We denote $$\Delta=\{(x,y)\in \I^2~~y=x\},~~\Delta^-=\{(x,y)\in \I^2~~y\leq x\}~~\text{and}~~\Delta^+=\{(x,y)\in \I^2~~y\geq x\}.$$

 \noindent For geometrical considerations, it suffices to discuss the following cases:
\begin{itemize}
\item[•] If $(v_1,v_2)$ and $(u_1,u_2)$ are both in the mean diagonal $\Delta$, meaning $u_1=v_1$ and $u_2=v_2$ then clearly  $V_P(R)=0$ holds
\item If $R\subset \Delta^-$
$$\begin{tabular}{ccc}
$V_P(R)$ & $=$ & $v_2(1-u_2)(u_2-v_2)+v_1(1-u_1)(u_1-v_1)-v_2(1-u_1)(u_1-v_2)-v_1(1-u_2)(u_2-v_1)$\\
$ $ & $=$ & $(v_2-v_1)(u_2-u_1)(1-u_2-u_1+v_1+v_2)$\\
$ $ & $=$ & $V_{\Pi}(R)(1-u_2-u_1+v_1+v_2)$\\
\end{tabular}$$
Thus $V_{C}(R)=V_{\Pi}(R)(2-u_2-u_1+v_1+v_2)\geq 0$
\item[•]If $R\subset \Delta^+$ then 
$$\begin{tabular}{ccc}
$V_P(R)$ & $=$ & $u_2(1-v_2)(u_2-v_2)+u_1(1-v_1)(u_1-v_1)-u_2(1-v_1)(u_1-v_2)-u_1(1-v_2)(u_2-v_1)$\\

\end{tabular}$$

\end{itemize} 
  It remains to prove that for all $u,v\in \I$ we  have $C(u,v)\in \I$.
 The partial derivatives of $C$ can be written as
follows: $$(u,v)\in \Delta^-\begin{cases} \frac{\partial C}{\partial u}(u,v)= v(2-2u+v)\\
\\
\frac{\partial C}{\partial v}(u,v)= 2u-2uv+v^2\end{cases}$$
 and 
$$(u,v)\in \Delta^+\begin{cases} \frac{\partial C}{\partial u}(u,v)= 2-2u-2v+2uv+v^2\\\\ \frac{\partial C}{\partial v}(u,v)= 1-2u-2v+2uv+u^2\end{cases}$$
Thus the extremal values of $C$ are taken in the borders of $\Delta^-$ and $\Delta^+$.
\end{proof}

%
 This new copula will allow to establish easily the third point of the following proposition
\begin{proposition}
\be 
\item $M$ is the unique copula $C$ satisfying $C_S=M$.
\item $W$ is the unique copula $C$ satisfying $C_S=W$.
\item There are many copulas $C$ satisfying $C_S=\Pi$.
\ee 
\end{proposition}
\begin{proof}
\begin{enumerate}
\item Let $C$ be a copula such  $C_S=M$. Meaning $C-M=M-C_S$. 
$C-M$ is non positive bifunction and equal to a non negative one $M-C_S$ thus $C-M=0$ 
\item Let $C$ be a copula such  $C_S=W$. Meaning $C-W=W-C_S$. 
$C-W$ is non negative bifunction and equal to a non positive  one $W-C_S$ thus $C=W$  
\item It is an immediate consequence of the theorem (\ref*{newcopula}) above.
\end{enumerate}
\end{proof}

\begin{theorem}\label{newcopula}
 For all $(u,v)\in \I^2$, we put: $Q(u,v)=(u+v-1)\min(1-u,v)\min(u,1-v)$. Then, the mapping
 $$C(u,v)=\Pi(u,v)+Q(u,v)=uv+(u+v-1)\min(1-u,v)\min(u,1-v)$$
 defines a copula.
\end{theorem}

\begin{proof}
$Q$ satisfies for all $u,v\in \I$ 
$$Q(1-u,1-v)=-Q(u,v), Q(u,0)=Q(u,1)=Q(0,v)=Q(1,v)=0$$ We denote  $$\Omega=\{(x,y)\in \I^2~~y=1-x\},~~,\Omega^-=\{(x,y)\in \I^2~~y\leq 1-x\},~\text{and}~\Omega^+=\{(x,y)\in \I^2~~y\geq1-x\}$$

We have to show that the independence perturbed copula $C=\Pi+Q$ is a copula.
\\The bifunction $C$ satisfies the border conditions.
 \\Let $R=[u_1,u_2]\times[v_1,v_2]$ be a rectangle in $\I^2$.  As above, the geometrical considerations lead to discuss only the following cases:
\begin{itemize}
\item[•] If $(u_1,v_2)$ and $(u_2,v_1)$ are in $\Omega$, then $$V_Q(R)=Q(u_1,v_1)+Q(u_2,v_2)-Q(u_1,1-u_1)-Q(u_2,1-u_2)=0$$

\item If $R\subset \Omega^-$
$$\begin{tabular}{ccc}
$V_Q(R)$ & $=$ & $v_2u_2(u_2+v_2-1)+u_1v_1(u_1+v_1-1)-u_1v_2(u_1+v_2-1)-u_2v_1(u_2+v_1-1)$\\
$ $ & $=$ & $V_{\Pi}(R)\left(u_1+u_2+v_1+v_2-1\right)$
\end{tabular}$$
Thus $V_{C}(R)=V_{\Pi}(R)(u_1+u_2+v_1+v_2)\geq 0$
\item[•]If $R\subset \Omega^+$. Let $R'$ be the rectangle $[1-u_2,1-u_1]\times[1-v_2,1-v_1]$ then 
$$\begin{tabular}{ccc}
$V_Q(R)$ & $=$ & $-V_Q(R')$\\
$ $ & $=$ & $-V_{\Pi}(R')\left(1-u_1+1-u_2+1-v_1+1-v_2-1\right)$\\
$ $ & $=$ & $-V_{\Pi}(R)\left(3-u_1-u_2-v_1-v_2\right)$\\
\end{tabular}$$
Thus $V_C(R)=V_{\Pi}(R)(u_1+u_2+v_1+v_2-2)\geq 0$ because $u_1+v_1\geq 1$ and $u_2+v_2\geq 1$

\end{itemize} 
 It remains to prove that for all $u,v\in \I$ we  have $C(u,v)\in \I$.
The partial derivatives of $C$ are given by
$$(u,v)\in \Omega^-\begin{cases} \frac{\partial C}{\partial u}(u,v)= v(2u+v)\\\\\frac{\partial C}{\partial v}(u,v)= u(u+2v)\end{cases}$$

and

$$(u,v)\in\Omega^+\begin{cases} \frac{\partial C}{\partial u}(u,v)= 2-2u-2v+2uv+v^2 \\\\ \frac{\partial C}{\partial v}(u,v)= 2-2u-2v+2uv+u^2\end{cases}$$
Thus the extremal values of $C$ are taken on the borders of $\Omega^-$ and $\Omega^+$.

\end{proof}
%

As just done for a symmetric copula, given a radial symmetric copula $R$. Does it exist a copula $C$ other than $R$ itself such that $C_R=R$? We try to answer for the same tree copulas as done above.
\begin{proposition}
\be 
\item $M$ is the only copula $C$ satisfying $C_R=M$.
\item $W$ is the only copula $C$ satisfying $C_R=W$.
\item There are many copulas $C$ satisfying $C_R=\Pi$.
\ee 
\end{proposition}

\begin{proof}
\begin{enumerate}
\item Let $C$ be a copula such  $C_R=M$. Meaning $C-M=M- \widehat{C}$. 
$C-M$ is non positive bifunction and equal to a non negative one $M-\widehat{C}$ thus $C=M$ 
\item Let $C$ be a copula such  $C_R=W$. Meaning $C-W=W-\widehat{C}$. 
$C-W$ is non negative bifunction and equal to a non positive  one $W-\widehat{C}$ thus $C=W$  
\ee
\end{proof}
\begin{remark}
For any given $\theta \in [-1,1]$,
\begin{itemize}
\item[•]The bifunction $~~\Pi+\theta P$ is a copula whose the symmetrized is $\Pi$.
\item[•] The bifunction $~~\Pi+\theta Q$ is a copula whose the radial symmetrized is $\Pi$.
\end{itemize}  

\end{remark}
\section{Conclusion}

The notion of symmetrized and radial symmetrized  copula will be probably a key to study the phenomena which are not symmetric by approximating with a symmetric distribution. We have cited some  properties preserved by the symmetrization procedure and we hope to investigate statistically the loss of such operation and characterize copulas for which the symmetrized one is archimedean.

\section*{Acknowledgment:}
\noindent The authors are grateful to anonymous referees for their prior remarks and effort.

\end{document}